\documentclass[sigconf,authorversion,nonacm]{acmart}

\usepackage{caption}
\usepackage{subcaption}
\usepackage[ruled, vlined]{algorithm2e}
\usepackage{tabularx}

\acmYear{2022} 
\setcopyright{acmcopyright}
\acmConference[KDD '22] {Proceedings of the 28th ACM SIGKDD Conference on Knowledge Discovery and Data Mining}{August 14--18, 2022}{Washington, DC, USA.}
\acmBooktitle{Proceedings of the 28th ACM SIGKDD Conference on Knowledge Discovery and Data Mining (KDD '22), August 14--18, 2022, Washington, DC, USA}
\acmPrice{15.00}
\acmISBN{978-1-4503-9385-0/22/08}
\acmDOI{10.1145/3534678.3539334}

\theoremstyle{definition}
\newtheorem{definition}{Definition}
\theoremstyle{remark}
\newtheorem*{remark}{Remark}
\newtheorem{theorem}{Theorem}
\newtheorem{lemma}[theorem]{Lemma}
\newtheorem{corollary}[theorem]{Corollary}
\newtheorem{observation}[theorem]{Observation}
\theoremstyle{proposition}
\newtheorem{proposition}{Proposition}

\newcommand{\REM}[1]{}
\usepackage{soul}

\def\R{{\mathbb{R}}}
\def\I{{\mathbb{I}}}
\def\UCaC{{\emph{UCaC}\xspace}}

\usepackage{xcolor}

\SetKwComment{Comment}{/* }{ */}

\DeclareMathOperator*{\argmax}{argmax} 

\begin{document}

\title{Solving the Batch Stochastic Bin Packing Problem in Cloud: \\ A Chance-constrained Optimization Approach}

\author{Jie Yan}
\authornote{Correspondence to Jie Yan and Qingwei Lin.}
\email{jiey@microsoft.com}
\affiliation{%
  \institution{Microsoft Research}
  \city{Beijing}
  \country{China}
}

\author{Yunlei Lu}
\authornote{Work done during the internship at Microsoft Research, Beijing.}
\email{luyunlei99@gmail.com}
\affiliation{%
  \institution{Peking University}
  \city{Beijing}
  \country{China}
}

\author{Liting Chen}
\authornotemark[2]
\email{98chenliting@gmail.com}
\affiliation{%
  \institution{Microsoft Research}
  \city{Beijing}
  \country{China}
}

\author{Si Qin}
\email{si.qin@microsoft.com}
\affiliation{%
  \institution{Microsoft Research}
  \city{Beijing}
  \country{China}
}

\author{Yixin Fang}
\email{yixfang@microsoft.com}
\affiliation{%
  \institution{Microsoft 365}
  \city{Suzhou}
  \country{China}
}

\author{Qingwei Lin}
\authornotemark[1]
\email{qlin@microsoft.com}
\affiliation{%
  \institution{Microsoft Research}
  \city{Beijing}
  \country{China}
}

\author{Thomas Moscibroda}
\email{moscitho@microsoft.com}
\affiliation{%
  \institution{Microsoft Azure}
  \city{Redmond}
  \state{WA}
  \country{USA}
}

\author{Saravan Rajmohan}
\email{saravar@microsoft.com}
\affiliation{%
  \institution{Microsoft 365}
  \city{Redmond}
  \state{WA}
  \country{USA}
}

\author{Dongmei Zhang}
\email{dongmeiz@microsoft.com}
\affiliation{%
  \institution{Microsoft Research}
  \city{Beijing}
  \country{China}
}


\renewcommand{\shortauthors}{Jie Yan et al.}

\begin{abstract} 
This paper investigates a critical resource allocation problem in the first party cloud: scheduling containers to machines. There are tens of services and each service runs a set of homogeneous containers with dynamic resource usage; containers of a service are scheduled daily in a batch fashion. This problem can be naturally formulated as {\em Stochastic Bin Packing Problem} (SBPP). However, traditional SBPP research often focuses on cases of empty machines, whose objective, i.e., to minimize the number of used machines, is not well-defined for the more common reality with nonempty machines. This paper aims to close this gap. First, we define a new objective metric, {\em Used Capacity at Confidence} (UCaC), which measures the maximum used resources at a probability and is proved to be consistent for both empty and nonempty machines, and reformulate the SBPP under chance constraints. Second, by modeling the container resource usage distribution in a generative approach, we reveal that \UCaC~can be approximated with Gaussian, which is verified by trace data of real-world applications. Third, we propose an exact solver by solving the equivalent cutting stock variant as well as two heuristics-based solvers -- \UCaC~best fit, bi-level heuristics. We experimentally evaluate these solvers on both synthetic datasets and real application traces, demonstrating our methodology's advantage over traditional SBPP optimal solver minimizing the number of used machines, with a low rate of resource violations.
\end{abstract}


\maketitle

\section{Introduction}
\label{sec:introduction}

Nowadays most large web companies run their services in containers~\cite{microservices, containerization}, and adopt Kubernetes-like systems~\cite{k8s, brendan2016, borg, borg-next} to orchestrate containers and manage resources in modern cloud platforms. One major challenge for the platform is container scheduling. In order to increase utilization, the platform consolidates multiple containers into a machine, where the sum of maximum resources required by containers may exceed the machine capacity. However, this also introduces risks of machine resource violations, which lead to container performance degradation or even service unavailability, resulting in financial loss for the application owner.

\begin{figure} 
  \centering
  \includegraphics[width=.95\linewidth]{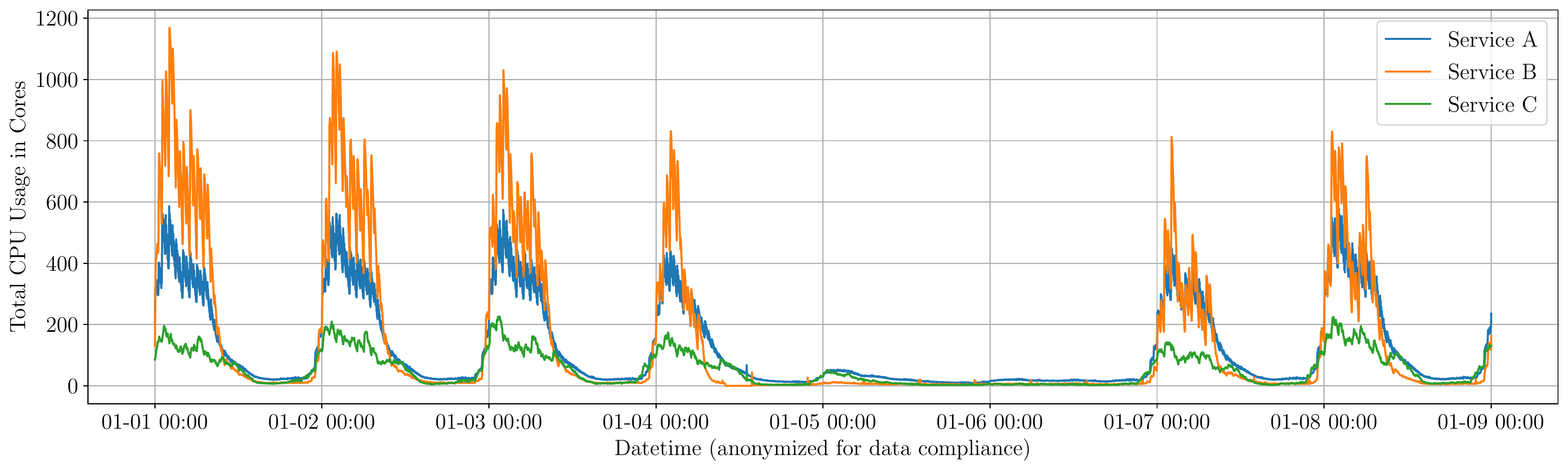}
  \caption{Total CPU core usage of three services in 8 days.}
  \label{fig:ts_cores}
\end{figure}

The container scheduling can be naturally modeled as the {\em Stochastic Bin Packing Problem} (SBPP) \citep{cohen2019, jeffrey2019} to seek an optimal resource utilization satisfying the constraint of keeping violations under a desired low level. To the best of our knowledge, previous research on SBPP assumes that all machines (also referred to as bins) are empty before allocation, and those approaches are evaluated in terms of the number of used bins. However, in practice, total resources required by a service change diurnally and weekly, as shown in Fig.~\ref{fig:ts_cores}. Thus, in order to increase resource utilization, the service often requests to allocate and delete a batch of containers every day. In the platform side, most machines often already have hosted a few containers when new allocation comes in. In the case when nonempty machines can host all or most requested containers, the previous metric, i.e., the number of used bins, fails to differentiate the goodness of allocation strategies. 

We propose a unified approach to the SBPP on both empty and nonempty machines. In particular, we consider the CPU resource utilization as the optimization objective and other types of resources (e.g., memory, storage) as constraints, since the main cloud providers charge for their resources by CPU core usage. As other optimization research work \citep{cohen2019, jeffrey2019}, we further focus on the SBPP of one type of resources (i.e., CPU). However, memory and storage resources as deterministic constraints can be easily processed if necessary. Our main contributions are summarized as follows:

$\bullet$ Reformulate the SBPP by introducing a new metric {\em Used Capacity at Confidence (UCaC)} as the optimization objective. {\em UCaC} is well-defined on both empty and nonempty machines. We prove in theory that for the SBPP on empty machines, minimizing {\em UCaC} is approximately consistent with minimizing the number of used machines.

$\bullet$ Approximate the chance constraints with Gaussian distributions. By modeling the workload process in a generative approach, we show that container resource usage distribution is naturally Gaussian when traffics are enough, and independent to other containers. Our analysis on traces of real web services empirically proves our assumptions.

$\bullet$ Propose three algorithms for the new {\em UCaC}-based SBPP formulation -- online best-fit heuristics, offline bi-level heuristics, and offline solver with cutting stock approach. 
Experiments with both synthetic data and real traces show that our {\em UCaC}-based algorithms perform better or equally well with respect to both {\em UCaC} and traditional metrics of used machines. In particular, when nonempty machines can host all or most containers in requests, our {\em UCaC}-based algorithms show a significant advantage.

\section{Preliminaries}
\label{sec:preliminaries}

\subsection{Container Scheduling in Cloud}
\label{sub:app-scenarios}

Our work targets general industrial infrastructures, such as Google's Borg\cite{borg, borg-next}, Kubernetes\cite{k8s}, and Microsoft AKS\cite{aks}. Without loss of generality, we use an abstract cluster with a minimum set of concepts. {\em Machine} is the bin that holds containers, and its resource capacity is a constant $V$. {\em Cluster} is the collection of $N$ machines. {\em Container} is the resource set (e.g., CPU and Memory) that a service instance is run on. For the container indexed by $j$, its resource usage is a random variable $A_j$. In our setting, the container size is dynamic and the same as its resource usage. We specify the maximum resource that a container can use as {\it limit}; once a container's resource usage reaches {\it limit}, it is throttled. {\em Service} is a container type, which runs a set of homogeneous containers. The number of services is denoted as $K$, and the service $k$ has $m_k$ containers. These concepts are also illustrated in Appendix (Fig.~\ref{fig:cluster}).

A cluster may run many services. Each service owner submits the request of allocating or deleting a number of containers to the cluster {\em scheduler}, which manages the set of machines and decides which containers should be on which machine. The scheduling algorithm is critical for resource utilization and is the target of this paper. We focus on the case that containers of the same service follow the same distribution of resource usage; If not, they should be further divided into more services. This is representative especially for first-party applications of web services (e.g., online meetings, news media, social networks).

\subsection{Chance Constraints}
\label{sub:chance-constraints}

To increase utilization, the scheduler often places multiple containers into one machine, where the sum of resources used by containers is required to not exceed the machine's capacity with confidence $\alpha$ whose value corresponds to a business SLA (Service Level Agreement)\cite{wieder2011service}. When solving the container scheduling problem, it is natural to define feasible solutions with {\em chance constraint}s. Formally, it is defined as follows: 
\begin{equation*}
\Pr\left(g(x; \theta) \leq c \right) \geq \alpha,
\end{equation*}
where $x$ is the decision variable, parameter $\theta$ is some prior random variable, and the constraint $g(x;\theta) \leq c$ satisfies with at least the probability $\alpha \in (0.0, 1.0]$.

When $g(x; \theta)$ follows a distribution that has low order moments as fully statistics, e.g., Gaussian, the chance constraint can be translated into a deterministic one in the form of mean plus uncertainties as $\overline{g}{(x;\theta)} + D(\alpha) \mathcal{U}_g(x;\theta) \leq c$. Here $D(\alpha)$ is a constant depending on $\alpha$. For Gaussian distributed $g(x)$, $D(\alpha) = \Phi^{-1}(\alpha)$ is the $\alpha$-quantile of the standard normal distribution.

\subsection{Stochastic Bin Packing Problem (SBPP)}
\label{sub:sbpp}

Container scheduling is modeled as a bin packing problem in our setting. Throughout this paper, machine and container are equivalent to bin and item in the bin packing problem respectively.

Suppose there are $M$ containers to be scheduled onto $N$ machines. We consider one type of resources -- CPU cores. Assume CPU resource usage of containers is random variables $\{ A_1, A_2, \dots, A_M\}$. The classic stochastic bin packing problem then is defined as the following 0-1 integer programming problem.
\begin{equation}
\label{eq:sbpp}
\begin{aligned}
  &\min_{x, y}{\sum_{i=1}^{N}{y_i}} \\
  s.t., \Pr( & \sum_{j=1}^M {A_j x_{i, j}} \leq V y_i) \geq \alpha, \quad i \in [N] \\
        & \sum_{i=1}^M {x_{i, j}} = 1, \quad j \in [M] \\
        & x_{i, j} \in \{0, 1\},  ~ y_i \in \{0, 1\}, \quad i \in [N], j \in [M] \\
\end{aligned}
\end{equation}
where the decision variable $x_{i,j}$ denotes whether or not container $j$ is  running on machine $i$, while $y_i$ denotes whether or not machine $i$ is used. Here $[N]$ denotes the set of integers $\{1, 2, \dots, N\}$.

Now we consider the chance constraints. Suppose the container's CPU usage is independent with others, and follows Gaussian distribution or, more loosely, a distribution with bounded support, as discussed in previous work~\cite{cohen2019}, the chance constraints in Eq.~\ref{eq:sbpp} can be translated into the following deterministic form.
\begin{equation}
\label{eq:deterministic-constraint-pre}
    \sum_{j=1}^{M} \mu_{j} x_{i,j}+D(\alpha) \sqrt{\sum_{j=1}^{M} b_{j} x_{i,j}} \leq V y_{i}, \quad x_{i, j} \in \{0, 1\}
\end{equation}
where $\mu_j$ is the mean of $A_j$, and $b_j$ is any uncertainty metric of $A_j$ (e.g., variance for Gaussian).

{\bf K-service SBPP.}~In the classic bin packing problem, the decision variable is for each pair of {\em(machine, container)}, leading to a decision space of $N \times M$ scale. There are often thousands of containers and hundreds of machines in reality, and the problem is quickly no longer solvable in a reasonable time using a commercial solver. In this paper, as stated before, we consider the scenario in which there are $K$ services, and containers of the same service are homogeneous and follow the same resource usage distribution.
Then the decision variable can be alternatively defined on {\em (machine, service)} pair, where the decision variable $x'_{i,j}$ denotes the number of containers of service $j$ placed on machine $i$. 
In reality, $K \ll M$, indicating that the SBPP is expected to be solved by commercial solvers. 
With the assumption that CPU usages of containers in the same service are iid (identical and independent distributed), as echoed in Observation \ref{obs:obv-1} and Observation \ref{obs:obv-2}, chance constraints and the SBPP formulation can be equivalently rewritten as follows:
\begin{equation}
\label{eq:sbp-k-services}
\begin{aligned}
\min_{x, y} & \sum_{i=1}^{N} y_{i} \\
\text { s.t. } & \sum_{j=1}^{K} \mu_{j} x_{i,j}+D(\alpha) \sqrt{\sum_{j=1}^{K} b_{j} x_{i,j}} \leq V y_{i}, \quad i \in [N] \\
& \sum_{i=1}^{N} x_{i,j} = m_j, \quad j \in [K] \\
& x_{i, j} \in \I, ~ y_{i} \in\{0,1\}, \quad i \in [N], ~j \in [K]
\end{aligned}
\end{equation}
where $x_{i,j}$ denotes the number of service $j$ running on machine $i$ and $y_i$ denotes whether or not machine $i$ is used. Here, $\I$ is the set of non-negative integers.

We mainly consider the Gaussian distributed container resource usage, since through broad analysis of real trace data, the empirical distribution is approximately Gaussian or Sub-Gaussian with truncation, as detailed in Section~\ref{sec:resource-usage-dist}. Note that, however, this assumption on distributions is not necessary for our work in Section~\ref{sec:formulation} and ~\ref{sec:methodology}.

\section{Modeling container resource usage}
\label{sec:resource-usage-dist}

We explore the trace data collected from a set of web application services running on a cluster. Results reveal that their CPU resource usage approximately follows the Gaussian distribution. Further, by developing a generative model to identify the internal cause process, we conclude that the assumption of Gaussian distribution should be common in most web services.

\subsection{Exploring Data Analysis on Real Traces}
\label{sub:eda}
We take three representative services as examples to analyze the distributions of container CPU usage at peak time. In Fig.~\ref{fig:dist_containers}, we compare the data distribution and fitted Gaussian distribution with the same mean and variance, in forms of the histogram and CDF (cumulative density function). As shown, the data distributions show the following characteristics.

$\bullet$ All three services are approximately Gaussian. A and B services have relatively large variances, while C service's distribution is very narrow. To the best of our experience, such diversity in distributions widely exists in real-world applications.

$\bullet$ The tails of data distributions decay faster than the fitted Gaussian distribution. Strictly speaking, empirical distributions are Sub-Gaussian, and their high quantifies are dominated by Gaussian. This means that the Gaussian assumption in our chance constraints formulation is more conservative than that of real data. 

$\bullet$ On the left side of the figures of service A and B, there are observed higher masses around zero. This is because the application manually reserves some container resources for tolerating accidental traffics.

$\bullet$ The tails of empirical distributions are bounded or truncated. This is because as depicted in section \ref{sub:app-scenarios}, one container can use resource in a range of $(0, \textit{limit})$ where {\it limit} is a per-service defined constant.

\begin{figure}
\centering
\includegraphics[width=.98\linewidth]{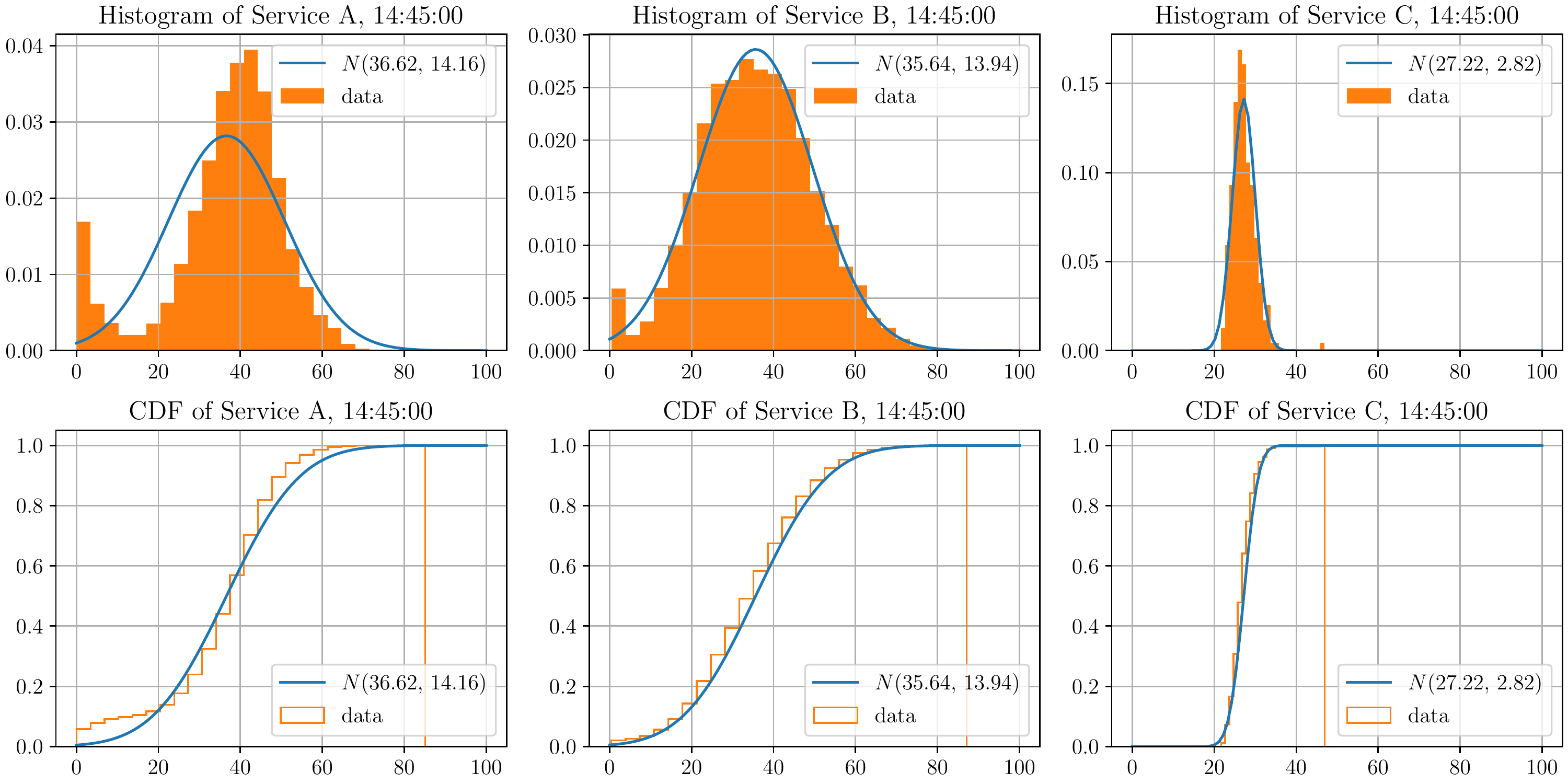}
\caption{Container CPU usage in percentage at a peak time.}
\label{fig:dist_containers}
\end{figure}

\subsection{Explanation in a Generative View}
\label{sub:resource-usage-dist}

Resource usage is a time series generated by a stochastic process driven by workloads assigned to the container. To solve SBPP, we are interested in the distribution of container resource usages {\em at the workload peak time}.

We use a generative view to analyze the container resource usage at a given time point. A container processes the assigned workloads of one service. For simplicity but without loss of generality, we assume the container resource usage is the sum of the assigned workloads' resource usage. Suppose that the workload of service $k$ follows the distribution $P_k$ and the container has $n$ such workloads. Then the container resource usage distribution $Y^{(n)}$ can be described as follows.
\begin{equation}
\label{eq:generative-usage}
\begin{aligned}
  & X_i \sim P_k,
  \quad &Y^{(n)} = \sum_{i=1}^{n}{X_i},
\end{aligned}    
\end{equation}
where $P_k$ is a distribution with $\mathbb{E}[X_i] = \mu$ and $Var[X_i]=\sigma^2 < \infty$.

\begin{observation}
\label{obs:obv-1}
At a given time $t$, resource usages of containers of the same service are independent and identically distributed (i.i.d).
\end{observation}

This observation holds in our situation because in service the workload scheduler allocates incoming workloads randomly to containers in round-robin or suchlike, which makes container usage distribution independent. Besides, since all containers of the same service are homogeneous, with roughly the same assigned workloads, their resource usage should be identical in distribution.

\begin{observation}
\label{obs:obv-2}
For a set of containers from different services, at a given time $t$, their resource usages are independently distributed.
\end{observation}

Similarly, because of the randomization in workload assignment, workloads of different containers are independent, and thus distributions of container resource usages.

\begin{corollary}[Gaussian distribution of container resource usage]
For Equation~\ref{eq:generative-usage}, $\lim_{n \to \infty} Y^{(n)} = \sum_{i=1}^{n}{X_i}$ approximately follows a Gaussian distribution.
\end{corollary}
\begin{proof}
Since $\{X_1, \ldots, X_n\}$ are independent and identically distributed random variables, by Central Limit Theorem we have $\frac{1}{n} Y^{(n)} \sim \mathcal{N}(\mu, \sigma^2/n)$. Thus, $Y^{(n)} \sim \mathcal{N}(n\mu, n\sigma^2)$.
\end{proof}


Note that at workload peaks, the number of workloads per container, $n$, is typically very large; thus, Gaussian distribution widely exists in our real-world traces. With bigger $n$, the container resource usage distribution is also more stationary. The support from the above generative mechanism, rather than only empirical statistics observed from limited data, makes methods based on Gaussian assumption perform more stable in reality.

\section{Problem Formulation}
\label{sec:formulation}
In the classic bin packing problem, the objective is to minimize the number of used machines. However, for cases with nonempty machines, the objective is no longer well-defined. 
To address this issue, we define a new metric {\em UCaC} and then reformulate the SBPP with the objective of minimizing {\em UCaC}. 
Moreover, we prove that with relaxed real-value decision variables, minimizing {\em UCaC} leads to the minimum number of used machines.
With integer decision variables, though minimizing {\em UCaC} does not necessarily lead to the minimum number of machines, we provide a constant factor guarantee from the minimum used machines.

\subsection{A New Metric: Used Capacity at Confidence}
\label{sub:sbp-metrics}

\begin{definition}
({\em UCaC}, Used Capacity at Confidence)
  Suppose on a machine there are $M$ containers sharing resources, and their resource usages are random variables $\{ X_i \sim P_i, i \in [M] \}$. The {\em UCaC} at confidence level $\alpha$ is defined as the minimum value of $U$ that satisfies $\Pr\left( \sum_{i=1}^{M}{X_i} \leq U \right) \geq \alpha$.
\end{definition}

For notation simplicity, we use the operator $U_{\alpha}(x)$ to denote the computation of {\em UCaC} of a vector of random variables $x = \{X_i \}$ at confidence $\alpha$. Suppose resource usages of the $M$ containers are Gaussian distributed, i.e., $X_i \sim N(\mu_i, \sigma_i^2$), then {\em UCaC} of the machine can be further written as follows:
$ U_{\alpha}(x) = \sum_{i=1}^{M} \mu_{i} + D(\alpha) \sqrt{\sum_{i=1}^{M} {\sigma_i^2}}.$

The new metric {\em UCaC} measures the maximum machine level resource usage at the confidence $\alpha$. 
For example, a machine runs three containers whose resource usages ${x_1, x_2, x_3}$ follow distributions $\mathcal{N}(2, 0.5), \mathcal{N}(2, 1)$, and $\mathcal{N}(3, 1.5) $ respectfully, with confidence level $\alpha =0.99$ ($D(0.99) = 2.576$), {\em UCaC} of this machine is calculated as:  $(2 + 2 + 3) + 2.576 * \sqrt{0.5 + 1 + 1.5} = 11.46$.

Naturally, we can extend to measure the resource usage of a cluster using {\em UCaC}, whether its machines are empty or nonempty.
\begin{definition}[{\em UCaC} of a cluster]
  A cluster's {\em UCaC} is the sum of {\em UCaC} of all its machines.
\end{definition}

Hence, {\em UCaC} is a well-defined metric for the SBPP with both empty and nonempty machines.

The final question is: what does lower {\em UCaC} mean in practice? The answer is the capability to serve more containers without increasing the number of nonempty machines in the cluster. 
Note that the estimated available capacity is the total capacity minus the {\em UCaC}.
Thus, less {\em UCaC} for existing containers means more capacity for future demands.


\subsection{Stochastic Bin Packing Reformulation}
\label{sub:sbp-reformulation}
Using the new objective {\em UCaC}, the SBPP on both empty and nonempty machines can be formulated as follows: 
\begin{equation}
\label{eq:initial}
\begin{aligned}
  \min_{x} ~& {\sum_{i=1}^{N}{ \left \{\sum_{j=1}^{K}{(z_{i,j} + x_{i, j})\mu_j} + D(\alpha) \sqrt{\sum_{j=1}^{K}{(z_{i, j} + x_{i, j}) b_j}}\right \}}}\\
  s.t. ~& \sum_{j=1}^{K}{(z_{i,j} + x_{i, j})\mu_j} + D(\alpha) \sqrt{\sum_{j=1}^{K}{(z_{i, j} + x_{i, j}) b_j}} \leq V, \\
       & \sum_{i=1}^{N}{x_{i, j}=m_j},  \quad x_{i, j} \in \I, \quad i \in [N], \quad j\in[K],
\end{aligned}
\end{equation}
where $z_{i, j}$ denotes initial number of containers of service $j$ on machine $i$. 
Given number of requested containers of service $j$, $m_j$, this formulation aims to minimize {\em UCaC} of the cluster after allocation while satisfying the capacity constraint at confidence level $\alpha$.

Note that the deterministic term $\sum_{i=1}^{N}{ \sum_{j=1}^{K}{(z_{i,j} + x_{i, j})\mu_j}}$ in the objective is a constant, thus minimizing {\em UCaC} is equivalent to minimizing the uncertainty term $\sum_{i=1}^{N}{D(\alpha) \sqrt{\sum_{j=1}^{K}{(x_{i, j}+z_{i, j}) b_j}}}$. 
With simple variable replacement and equivalence transformation, we get the following final problem formulation.
\begin{equation}
\label{eq:reformulation}
\begin{aligned}
  \min_{x} ~& \sum_{i=1}^{N}{D(\alpha) \sqrt{\sum_{j=1}^{K}{x_{i, j} b_j} + B_i}}\\
  s.t. ~& \sum_{j=1}^{K}{x_{i, j}\mu_j} + D(\alpha) \sqrt{\sum_{j=1}^{K}{x_{i, j} b_j} + B_i} \leq V - C_i, \quad i \in [N] \\
       & \sum_{i=1}^{N}{x_{i, j}=m_j},  \quad x_{i, j} \in \I, \quad i \in [N], \quad j\in[K],
\end{aligned}
\end{equation}
where $B_j = \sum_{j=1}^{K}{z_{i, j} b_j}$ and $C_i = \sum_{j=1}^{K}{z_{i, j} \mu_j}$.

We can prove that for the SBPP on empty machines, our {\em UCaC}-based formulation is consistent with traditional SBPP that optimizes the number of used machines.

\begin{theorem}[Relaxed Consistency] 
\label{th:relaxed-consistency}
Relax decision variable $x_{i,j}$ to real value, and suppose that $b_j > 0, ~\forall j \in [K], i \in [N]$ in Eq.~\ref{eq:reformulation} and Eq.~\ref{eq:sbp-k-services}. 
Then, minimizing {\em UCaC} leads to the least number of used bins in stochastic bin packing on empty machines.
\end{theorem}

\begin{proposition}
\label{th:marginal-decrease}
Consider the function $y = \sqrt{x}$ with $\text{\em dom}(x) = \R^+$. Since $\frac{\partial y}{\partial x} = \frac{1}{2\sqrt{x}}$ is monotonically decreasing, the marginal change $\Delta y$ by $\Delta x$ decreases. That means for any $0 < x_1 < x_2$ and $\Delta x > 0$, $\sqrt{x_1 + \Delta x} - \sqrt{x_1} > \sqrt{x_2 + \Delta x} - \sqrt{x_2}$ holds.
\end{proposition}

\begin{proof}[Proof of Theorem~\ref{th:relaxed-consistency}]

We first prove that \textit{at the optimal UCaC, there is at most one machine is not full.} Otherwise, we can always find two non-full machines, say $i$ and $i'$. Suppose their uncertainty terms satisfy $\sqrt{\sum_{j=1}^{K}{x_{i,j} b_j}} \leq \sqrt{\sum_{j=1}^{K}{x_{i',j} b_j}}$. Then by Proposition~\ref{th:marginal-decrease}, after moving some amount of $i$'s container to $i'$, say $\Delta x_{i,j'}b_{j'}$, we have 
$
\sqrt{\sum_{j=1}^{K}{x_{i,j} b_j - \Delta x_{i,j'} b_{j'}}} + \sqrt{\sum_{j=1}^{K}{x_{i',j} b_j} + \Delta x_{i,j'} b_{j'}} < \sqrt{\sum_{j=1}^{K}{x_{i,j} b_j}} + \sqrt{\sum_{j=1}^{K}{x_{i',j} b_j}}
$
holds; given that $\sum_{j=1}^{K}{x_{i, j}\mu_j} + D(\alpha) \sqrt{\sum_{j=1}^{K}{x_{i, j} b_j} + B_i} < V - C_i$, it is always possible to find such $\Delta x_{i,j'} $ that satisfies
$
\sum_{j=1}^{K}{x_{i, j}\mu_j}+\Delta x_{i,j'} \mu_{j'} + D(\alpha) \sqrt{\sum_{j=1}^{K}{x_{i, j} b_j} +\Delta x_{i,j'}b_{j'} + B_i} \leq  V - C_i.
$
That means {\em UCaC} of this cluster can be reduced, which contradicts our assumption of optimal $UCaC$.

We then prove that \textit{at the optimal UCaC, the number of used machines is equal to the least used machines.} Otherwise, we assume the optimality of Eq.~\ref{eq:sbp-k-services} is $n_1$ machines and  the optimal solution  of Eq.~\ref{eq:reformulation} corresponds to $n_2$ machines. By definitions, $n_1 \leq n_2$ and $U_1 \geq U_2$ holds, where $U_1$ and $U_2$ are {\em UCaC} of the optimal solutions of Eq.~\ref{eq:sbp-k-services} and \ref{eq:reformulation} in respect. 
Now we assume $n_1 < n_2$. Let $\Delta_1$ and $\Delta_2$ be the unused capacity of th solutions corresponding to optimal used bins  and UCaC. Since there is at most one machine is not full, $ (V-\Delta_2)>0$. Then $U_2 - U_1 = (n_2 V  - \Delta_2) - (n_1 V  - \Delta_1) = (n_2 - n_1) V - (\Delta_2 - \Delta_1) > (V - \Delta_2) + \Delta_1 > 0$ which contradicts the fact of $U_1 \geq U_2$. Thus, $n_1 = n_2$.
\end{proof}

For the integer type decision variables, however, minimizing {\em UCaC} does not necessarily lead to a minimum number of used machines, especially in the extreme case when the ratio of variance against the mean of the resource usage $A_j$ is zero. 
The following theorem guarantees that the number of used machines corresponding to {\em UCaC} minimized can be bounded by the minimum number of used machines.

\begin{theorem}[Bounded Number of Used Machines] 
\label{th:bound}
The optimal solution of Eq.~\ref{eq:reformulation} that minimizes {\em UCaC} on empty machines uses at most $\frac{8}{3}N^*$ machines, where $N^*$ is the optimal number of used machines by solving Eq.~\ref{eq:sbp-k-services}.
\end{theorem}
The proof is given in Appendix \ref{app:proof-of-th5}.

\section{Methodology}
\label{sec:methodology}

In this section, we propose three algorithms for the SBPP that optimizes {\em UCaC} -- online Best-Fit heuristic, offline bi-level heuristic and offline suboptimal solver using a cutting stock approach.

\subsection{Online Best-Fit Heuristic Algorithm}
\label{sub:solve-online-heuristic}

{\em Best-Fit} is a greedy algorithm for the bin packing problem. For the SBPP that optimizes cluster {\em UCaC}, the Best-Fit heuristic allocates the new container to the machine that is with maximum {\em UCaC} after allocation and can fit the container. If no such machine exists, the new container is allocated to a new empty machine. Let $x_i$ be the allocation on machine $i$ and $\Delta x$ denote the new container, 
the Best-Fit heuristic selects machine $i^* = \argmax_{i \in [N], U_\alpha (x_i + \Delta x) \leq V} \{ U_\alpha (x_i + \Delta x)  \}$ for the new container if such machine exists.

In classic deterministic bin packing, the intuition of Best-Fit is to fit containers with the least resulted fragment at the decision time. It can be concluded that in our SBPP formulation, Best-Fit greedily minimizes {\em UCaC}.
In the online setting where container requests are processed sequentially, the Best-Fit heuristic is one of the most robust and efficient methods in practice to our knowledge.

\subsection{Offline Heuristic Algorithm} 
\label{sub:solve-offline-heuristic} 

In the cloud applications, when a batch of containers of multiple services needs to be allocated simultaneously, the information of containers in the batch is known at decision time. Thus, offline algorithms can be utilized since they generally have better performance than online algorithms. We therefore propose a bi-level offline heuristic (Algorithm \ref{alg:solver-bilevel-heuristics}) in this subsection and a sub-optimal offline algorithm with better optimality in Section \ref{sub:solve-offline-cuttingstock}.

The details of the proposed bi-level heuristic is given in Appendix (Alg.~\ref{alg:solver-bilevel-heuristics}). The main intuitions are as follows.
\begin{itemize}
    \item Sort machines by cumulative uncertainty $\{B_i\}$ non-increasingly. The
    machine with larger \textit{uncertainty} term $B_i$ is preferred,
    since {\em UCaC} is a submodular function and allocating new containers on the machine with larger $B_i$ minimizes cluster {\em UCaC}.
    The underlying motivation is the same with the best-fit heuristics.
    \item Sort containers by metric $\{{b_j}/{\mu_j}\}$ non-increasingly.
    Container with larger \textit{normalized uncertainty} (${b_j}/{\mu_j}$) is preferred,
    since allocating containers with larger \textit{normalized uncertainty} greedily optimizes cluster {\em UCaC}.
\end{itemize}
The insights above are motivated by the proof of Theorem~\ref{th:relaxed-consistency}.


\subsection{Offline Algorithm: Cutting Stock Approach}
\label{sub:solve-offline-cuttingstock}

In the case where the number of services $K$ is not too large, and there is enough historical data of different services, the bin packing problem with empty machines can be efficiently solved as a cutting stock problem (CSP, detailed in Appendix \ref{app:solve-sbp-empty}).
In this subsection, we propose a cutting stock formulation of the SBPP for nonempty machines and a corresponding offline algorithm (Algorithm \ref{alg:solver-cutting-stock}).

\subsubsection{The  CSP formulation}
The CSP is solved by generating patterns iteratively.
A \textit{pattern} $p_j\in \R^K$ is a combination of containers in each service that satisfies the constraints
\begin{align}
& \sum_{k=1}^{K} \mu_{k} p_{kj} + D(\alpha) \sqrt{\sum_{k=1}^{K} b_{k} p_{kj}} \leq V, ~ p_{kj}\in \I, \quad k \in [K] \label{eq:pattern1}.
\end{align}
Constraint (\ref{eq:pattern1}) guarantees that the pattern is feasible.
For the SBPP with empty bins, one can optimize the continuous relaxation of the CSP (\ref{eq:cutting-stock-relax}) by finding a pattern set $P\in\R^{K\times L}$ using the column generation method (detailed in Appendix \ref{app:solve-sbp-empty}). 
The column generation method is a procedure that adds a single pattern (column) to $P$ at each iteration until no patterns can be generated by this method.

The following proposition shows that given a proper pattern set, the SBPP with either nonempty or empty bins can be efficiently solved by the CSP formulation (\ref{eq:cutting-stock-generalized}) that optimizes {\em UCaC}.
\begin{proposition}
\label{proposition:non-empty-cutting-stock}
  Given service distributions $(\mu_k, b_k)$,  requests $m_k$ for $k \in [K]$, and a  pattern set $P$, 
  suppose containers in all nonempty machines can be covered by some patterns in $P$, i.e., the number of containers of each service in each nonempty machine is not larger than the number of containers of this pattern,  then the cluster UCaC can be optimized by optimizing the  following cutting stock problem:
  \begin{equation}
    \label{eq:cutting-stock-generalized}
    \begin{aligned}
    \min_{v, w} ~& \sum_{j \in [P]} v_j u_j = \text{{\em UCaC}} \\
    \text{s.t.} ~& (p_j - z_i) w_{ij} \geq 0,  \quad i, j \in [N]\times [P] \\
            & \sum_{j \in [P]} p_{k j} v_j \geq m_k + \sum_{i\in[N]} z_{ik}, \quad k \in [K]\\
            & \sum_{i=1}^{n} w_{ij} = v_j,  \quad j \in [P] \\
            & \sum_{j \in [P]} w_{ij} = 1,  \quad i \in [N_0] \\
            & \sum_{j \in [P]} w_{ij} \leq 1,  \quad i \notin [N_0] \\
            & w_{ij} \in \{0, 1\}, \quad v_j \in \I,
    \end{aligned}
  \end{equation}
  where $v_j$ is the number of times pattern $p_j$ is used, $u_j$ is the UCaC of pattern $p_j$, 
  $w_{ij} = 1$ only if pattern $p_j$ is used in machine $i$, 
  $[n_0]$ is the index set of nonempty machines and $[P]$ is the index set of patterns.
\end{proposition}

\subsubsection{Solving the CSP}
We now propose the framework to solve the CSP (\ref{eq:cutting-stock-generalized}).
The proposed Algorithm \ref{alg:solver-cutting-stock} consists of three separated parts: generating a feasible pattern set (\texttt{PatternGen}), solving the generalized CSP (\texttt{SolveCSP}) and compute the container placement.
\begin{algorithm}
\DontPrintSemicolon
\caption{CSP Algorithm for SBPP}
\label{alg:solver-cutting-stock}
\SetAlgoVlined
\SetKwInOut{Input}{Input}
\SetKwInOut{Output}{Output}
\SetKw{KwBy}{by}
\Input{Request $m \in \I^K$}
\Input{Parameters of Equation \ref{eq:reformulation}: $\{ \{(\mu_j, b_j) \}, z, V, \alpha\}$}
\Output{Mapping of containers to machines $x \in I^{N \times K}$}
\sf{
\Begin{
    \tcp{Pattern Generation}
    \nl $P \gets $ PatternGen ($ \{(\mu_j, b_j) \}, z, V, \alpha $); \\
    
    \tcp{Solve the CSP (\ref{eq:cutting-stock-generalized})}
    \nl   $w \gets $ SolveCSP $\left(\{(\mu_j, b_j) \}, z, V, \alpha, P \right) $;\\
    
    \tcp{Compute Container Placement}
    \nl $x \gets $ PatternCombinationToContainerPlacement ($w, P$); \\ 
} 
} 
\end{algorithm}

The first part of the algorithm (\texttt{PatternGen}) can be computed ahead of the container resource usage peaks, 
since generating a pattern set by the column generation method can be computationally expensive compared to solving the CSP.

\texttt{SolveCSP} solves the generalized CSP (\ref{eq:cutting-stock-generalized}). 
Note that the optimization problem is a  mixed-integer linear programming problem, which can be efficiently solved by commercial solvers that give better sub-optimal solutions than online or offline heuristics. 

In the end, the container placement $x_{ij}$ for new requests is computed from $w_{ij}$ output by \texttt{SolveCSP}.
Thus, $x_{i} = p_j - z_i$ when pattern $p_j$ is used to cover machine $i$.

\section{Experiments}
\label{sec:experiments}
In this section, we evaluate the performance of our proposed methods for the {\em UCaC}-based SBPP formulation on both synthetic datasets and real cloud traces. We refer to the three algorithms in Section \ref{sec:methodology} as: {\bf BF-UCaC} (best-fit with {\em UCaC}), {\bf BiHeu} (offline bi-level heuristics) and {\bf CSP-UCaC} (solving cutting stock problem minimizing {\em UCaC}). We compare with two baselines: {\bf BF-n$\sigma$} and {\bf CSP-Mac}. BF-n$\sigma$ is the deterministic Best-Fit algorithm, in which the container size is a deterministic value estimated by the n-Sigma method ($n = \Phi^{-1}(\alpha)$). CSP-Mac minimizes the number of used {\bf machines}, and is implemented in the same algorithmic framework as CSP-UCaC.

In CSP-UCaC and CSP-Mac, Gurobi 9.5~\cite{gurobi} is used to solve the optimization problems in Section \ref{sub:solve-offline-cuttingstock}. 
For all experiments, we use a simulated cluster where 4,000 machines run on and machine capacity is 31.58 CPU cores. 

\subsection{Experiments on Synthetic Datasets}
\label{sub:experiments-sync}

The synthetic datasets are generated as follows. (1) The per-service container resource usage distributions are sampled from a pool of 17 Gaussian distributions summarized from real traces and then multiplying the standard deviation with a random factor, i.e., $(\mu, \sigma' = \sigma \times \mathcal{U}(0.9, 1.1))$. 
(2) The instances of container resource usages are sampled by the per-service distributions independently and then truncated to $(0, \text{limit})$, which are used to estimate the violations.
(3) The nonempty container layout is generated like this: stochastic bin packing on empty machines and then randomly delete about half containers. 
(4) The requests are generated by computing the difference between the target container layout -- a predefined initial layout multiplying a scale factor to each service container number, and the current layout.

\subsubsection{Experiments on nonempty cluster}

We test two cases on non-empty cluster layout, named scale-up and scale-down.
The nonempty cluster layout (i.e., the mapping of containers to machines) is constructed as follows. We first randomly generate the per-service container resource usage distributions.
Then, we build an {\em initial layout} by solving with bf CSP-UCaC. 
Finally, we evict part of the containers from the layout according to the rate given in Table~\ref{tab:synthetic_stats}. 

In the scale-down (scale-up) case, after allocating the requested containers, the number of containers is smaller (larger) than that of the {\em initial layout}. 
To make results with different initial layouts comparable, we normalize the {\em UCaC} and the number of used machines (\#Mac) with the value of BF-n$\sigma$.
We perform 5 runs of tests with different random seeds and report the average values. \\
\textbf{Results:} ~ 
Table~\ref{tab:perf-scale-down} and \ref{tab:perf-scale-up} show detailed performance comparison for different number of services $K$ and confidence level $\alpha$ in the scale-down and scale-up cases.

Compared to the deterministic best fit BF-n$\sigma$, our three {\em UCaC}-based algorithms have lower {\em UCaC} and less \#machines in every test case.
Although the violation rates are higher for the three algorithms, they are still less than the given risk $1-\alpha$ in most cases, showing less resource usage with a satisfying confidence level.
\footnote{For rare cases in $K=10$ and $K=15$, the violation rate is more than $1-\alpha$; this is because of the container resource usage truncation, which leads to a minor positive bias of empirical mean and variances from that of the sampling distributions.} 
The results also show that our proposed algorithms have similar performance over different test cases.

Compared to CSP-Mac, which optimizes \#machines, our proposed methods have similar performance on {\em UCaC} and used machines for the scale-up case (Table~\ref{tab:perf-scale-up}).
However, for the scale-down case (Table~\ref{tab:perf-scale-down}) in which all new requests can be allocated into existing nonempty machines, our proposed methods achieve lower {\em UCaC}, indicating lower resource usage on the cluster.
The results in the scale-down case demonstrate the effectiveness of the metric {\em UCaC} and the {\em UCaC}-based SBPP formulation.

\begin{table}
\centering
\scalebox{0.75} 
{
\begin{tabular}{ccccc|ccc}%
\toprule
& & \multicolumn{3}{c}{$\alpha = 99.9\%$} & \multicolumn{3}{c}{$\alpha = 99\%$}  \\
\cmidrule(lr){3-5} \cmidrule(lr){6-8} 
$K$   & Alg.         & UCaC & \#Mac & Violation (\%) & UCaC & \#Mac & Violation (\%) \\
\midrule
& BF-n$\sigma$ & 1.00 & 1.00 & 0.00 & 1.00 & 1.00 & 0.03 \\
& BF-UCaC & 0.94 & 0.71 & 0.04 & 0.96 & 0.75 & 0.47 \\
5 & BiHeu & 0.94 & 0.71 & 0.02 & 0.96 & 0.75 & 0.62 \\
& CSP-UCaC & 0.94 & 0.71 & 0.02 & 0.96 & 0.75 & 0.46 \\
& CSP-Mac & 1.02 & 0.71 & 0.06 & 1.04 & 0.75 & 0.90 \\
\midrule
& BF-n$\sigma$ & 1.00 & 1.00 & 0.00 & 1.00 & 1.00 & 0.01 \\
& BF-UCaC & 0.94 & 0.70 & 0.04 & 0.96 & 0.73 & 0.52 \\
10 & BiHeu & 0.93 & 0.70 & 0.07 & 0.95 & 0.73 & 0.79 \\
& CSP-UCaC & 0.94 & 0.70 & 0.07 & 0.95 & 0.73 & 0.80 \\
& CSP-Mac & 1.01 & 0.70 & 0.06 & 1.04 & 0.73 & 0.82 \\
\midrule
& BF-n$\sigma$ & 1.00 & 1.00 & 0.00 & 1.00 & 1.00 & 0.02 \\
& BF-UCaC & 0.94 & 0.70 & 0.06 & 0.96 & 0.76 & 0.69 \\
15 & BiHeu & 0.93 & 0.70 & 0.04 & 0.95 & 0.76 & 0.57 \\
& CSP-UCaC & 0.93 & 0.70 & 0.11 & 0.95 & 0.76 & 0.26 \\
& CSP-Mac & 1.0 & 0.70 & 0.15 & 1.03 & 0.76 & 0.57 \\
\midrule
& BF-n$\sigma$ & 1.00 & 1.00 & 0.00 & 1.00 & 1.00 & 0.00 \\
& BF-UCaC & 0.94 & 0.74 & 0.04 & 0.96 & 0.77 & 0.51 \\
20 & BiHeu & 0.93 & 0.74 & 0.07 & 0.95 & 0.77 & 0.69 \\
& CSP-UCaC & 0.94 & 0.74 & 0.02 & 0.96 & 0.77 & 0.26 \\
& CSP-Mac & 1.01 & 0.74 & 0.00 & 1.05 & 0.77 & 0.53 \\
\bottomrule
\end{tabular}
}
\caption{Performance comparison in the scale-down case. 
}
\label{tab:perf-scale-down}
\end{table}

\begin{table}
\centering
\scalebox{0.75} 
{
\begin{tabular}{ccccc|ccc}%
\toprule
& & \multicolumn{3}{c}{$\alpha = 99.9\%$} & \multicolumn{3}{c}{$\alpha = 99\%$}  \\
\cmidrule(lr){3-5} \cmidrule(lr){6-8} 
$K$   & Alg.         & UCaC & \#Mac & Violation (\%) & UCaC & \#Mac & Violation (\%) \\
\midrule
& BF-n$\sigma$ & 1.00 & 1.00 & 0.00 & 1.00 & 1.00 & 0.00 \\
& BF-UCaC & 0.94 & 0.66 & 0.06 & 0.96 & 0.70 & 0.79 \\
5 & BiHeu & 0.94 & 0.65 & 0.07 & 0.95 & 0.69 & 0.60 \\
& CSP-UCaC & 0.93 & 0.64 & 0.06 & 0.95 & 0.68 & 0.97 \\
& CSP-Mac & 0.93 & 0.64 & 0.09 & 0.95 & 0.68 & 0.85 \\
\midrule
& BF-n$\sigma$ & 1.00 & 1.00 & 0.00 & 1.00 & 1.00 & 0.00 \\
& BF-UCaC & 0.93 & 0.64 & 0.03 & 0.95 & 0.67 & 0.76 \\
10 & BiHeu & 0.93 & 0.64 & 0.08 & 0.95 & 0.67 & 0.79 \\
& CSP-UCaC & 0.92 & 0.63 & 0.13 & 0.94 & 0.66 & 0.86 \\
& CSP-Mac & 0.92 & 0.63 & 0.14 & 0.94 & 0.66 & 1.13 \\
\midrule
& BF-n$\sigma$ & 1.00 & 1.00 & 0.00 & 1.00 & 1.00 & 0.00 \\
& BF-UCaC & 0.93 & 0.66 & 0.07 & 0.95 & 0.69 & 0.66 \\
15 & BiHeu & 0.93 & 0.65 & 0.08 & 0.94 & 0.68 & 0.80 \\
& CSP-UCaC & 0.92 & 0.64 & 0.12 & 0.95 & 0.69 & 0.42 \\
& CSP-Mac & 0.92 & 0.64 & 0.12 & 0.95 & 0.69 & 0.56 \\
\midrule
& BF-n$\sigma$ & 1.00 & 1.00 & 0.00 & 1.00 & 1.00 & 0.00 \\
& BF-UCaC & 0.94 & 0.66 & 0.03 & 0.95 & 0.69 & 0.52 \\
20 & BiHeu & 0.93 & 0.65 & 0.00 & 0.95 & 0.68 & 0.96 \\
& CSP-UCaC & 0.94 & 0.68 & 0.07 & 0.95 & 0.70 & 0.70 \\
& CSP-Mac & 0.94 & 0.68 & 0.03 & 0.96 & 0.70 & 0.51 \\
\bottomrule
\end{tabular}
}
\caption{Performance comparison in the scale-up case.
}
\label{tab:perf-scale-up}
\end{table}

\subsubsection{Experiments on empty cluster}
We also test on the empty cluster.
As shown in Theorem~\ref{th:relaxed-consistency} and ~\ref{th:bound}, optimizing {\em UCaC} induces optimizing the number of used machines. 
Thus, we compare the performance of CSP-UCaC and CSP-Mac.\\
\textbf{Results:} ~ 
The results are shown in Table \ref{tab:relaxed-consistency}. 
Compared to CSP-Mac, CSP-UCaC has lower {\em UCaC} in every test case, and \#machines are the same or slightly larger as expected.
This empirically verifies that for empty cluster machines, our {\em UCaC}-based methods are consistent with the SBPP that optimizes the number of used machines.
For the violation rates, the two methods have similar performance and satisfy the given confidence levels in most cases.

\begin{table}
\centering
\scalebox{0.75} 
{
\begin{tabular}{ccccc|ccc}%
\toprule
& & \multicolumn{3}{c}{$\alpha = 99.9\%$} & \multicolumn{3}{c}{$\alpha = 99\%$}  \\
\cmidrule(lr){3-5} \cmidrule(lr){6-8} 
$K$   & Alg.         & UCaC & \#Mac & Violation (\%) & UCaC & \#Mac & Violation (\%) \\
\midrule
5 & CSP-UCaC & 32360 & 1035 & 0.1 & 29604 & 942 & 0.9 \\
& CSP-Mac & 32366 & 1035 & 0.1 & 29614 & 941 & 1.1 \\
\midrule
10 & CSP-UCaC & 33156 & 1056 & 0.2 & 30160 & 958 & 0.8 \\
& CSP-Mac & 33177 & 1054 & 0.1 & 30163 & 957 & 1.0 \\
\midrule
15 & CSP-UCaC & 33708 & 1070 & 0.1 & 30670 & 1008 & 0.5 \\
& CSP-Mac & 33713 & 1069 & 0.1 & 30721 & 1005 & 0.5 \\
\midrule
20 & CSP-UCaC & 33413 & 1118 & 0.0 & 30255 & 995 & 0.4 \\
& CSP-Mac & 33456 & 1110 & 0.1 & 30295 & 994 & 0.7 \\
\bottomrule
\end{tabular}
}
\caption{Performance of bin packing on the empty cluster.}
\label{tab:relaxed-consistency}
\end{table}

\subsubsection{Comprehensive comparison}
Fig.~\ref{fig:bp-visual} shows the machine {\em UCaC} usages visualization of the solution given by different algorithms in the nonempty and empty cluster cases.  
The results are from the same initial cluster state with $K=20$ and $\alpha=0.999$.

As shown in Fig.~\ref{subfig:scale-down}, BF-n$\sigma$ uses much more machines than others, although its per-machine {\em UCaC} is lower.
For CSP-Mac, almost all used machines have higher {\em UCaC}, as the algorithm fails to optimize the resource usage.
In contrast, our three {\em UCaC}-based algorithms can effectively optimize both the {\em UCaC} in both nonempty and empty machines cases, as optimizing {\em UCaC} is able to take full advantage of the risk-pooling effect to optimize the cluster resource usage.

As shown in Fig.~\ref{subfig:scale-up} and Fig.~\ref{subfig:empty}, the CSP-UCaC and CSP-Mac show similar machine {\em UCaC} usage distributions in the scale-up case of nonempty cluster and the case of empty cluster.
This indicates that when new machines are needed to allocate containers, optimizing {\em UCaC} and \#machines leads to close solutions. 
In conclusion, our proposed {\em UCaC}-based problem formulation and solving algorithms give overall better performance on all three cases.

\begin{figure}
    \begin{subfigure}{0.4\textwidth}
      \centering
      \includegraphics[width=.9\linewidth]{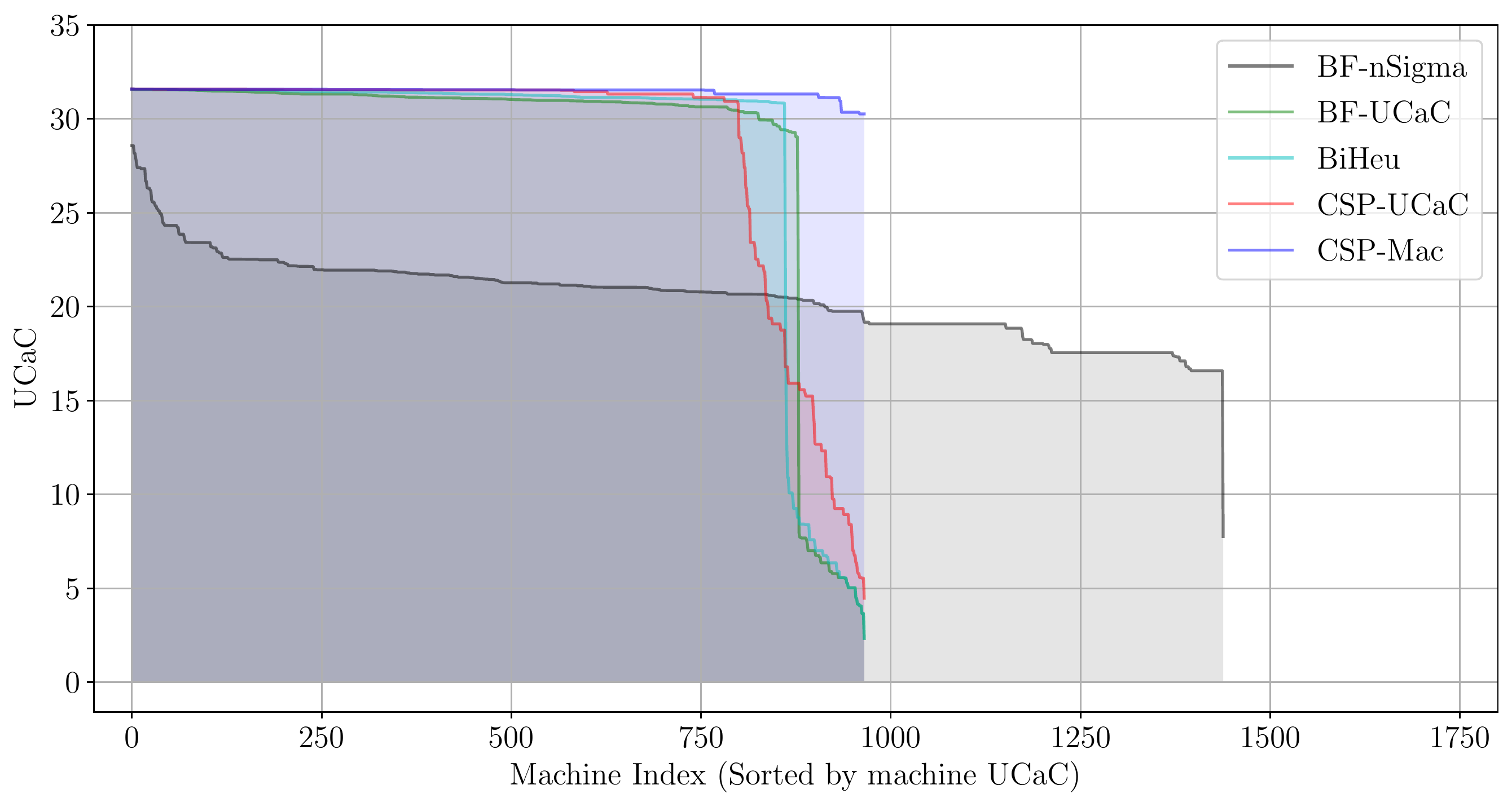}
      \caption{Nonempty cluster case: Scale-down}
      \label{subfig:scale-down}
    \end{subfigure}%
    \\
    \begin{subfigure}{0.4\textwidth}
      \centering
      \includegraphics[width=.9\linewidth]{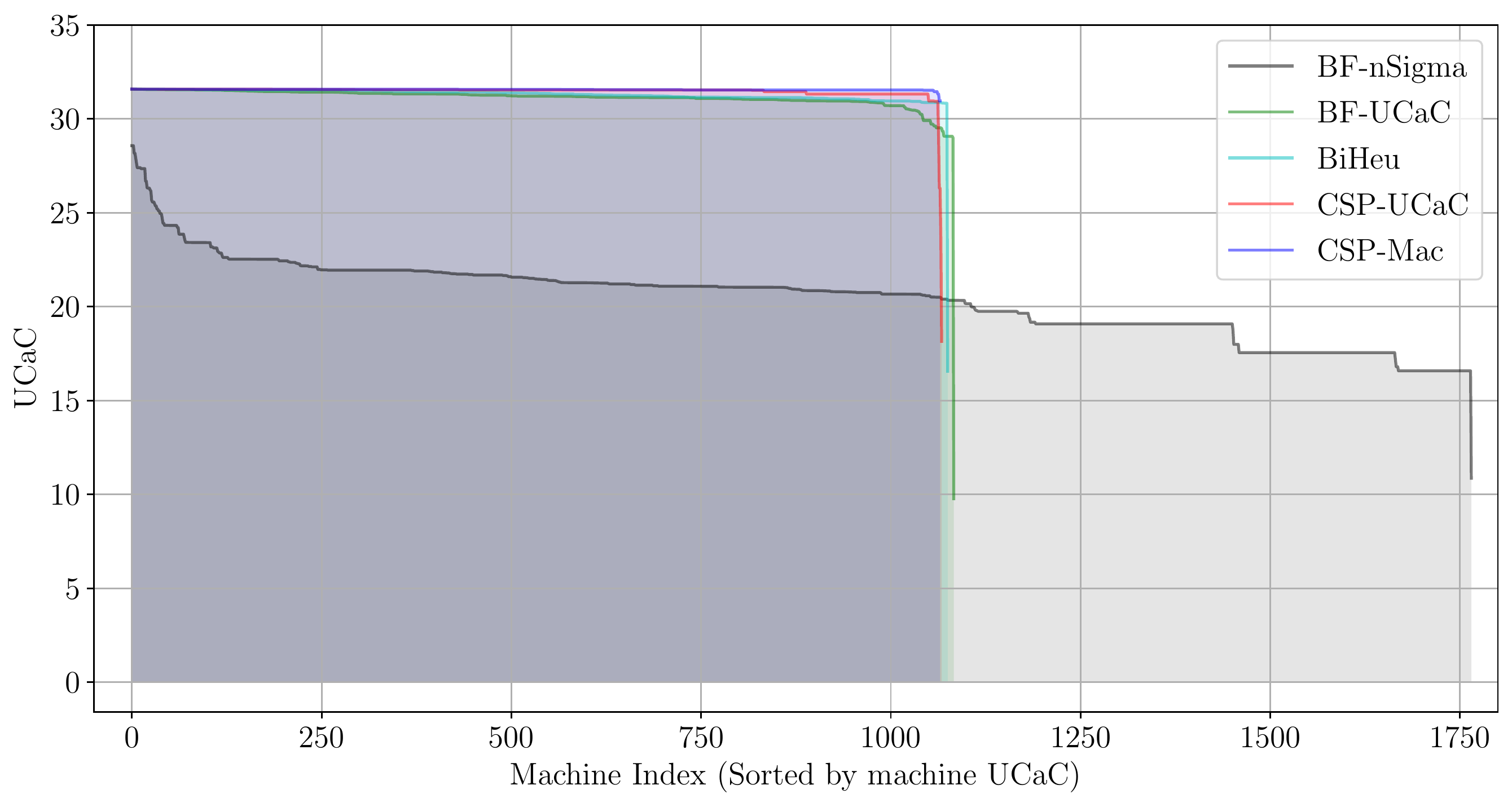}
      \caption{Nonempty cluster case: Scale-up}
      \label{subfig:scale-up}
    \end{subfigure}%
    \\
    \begin{subfigure}{0.4\textwidth}
      \centering
      \includegraphics[width=.9\linewidth]{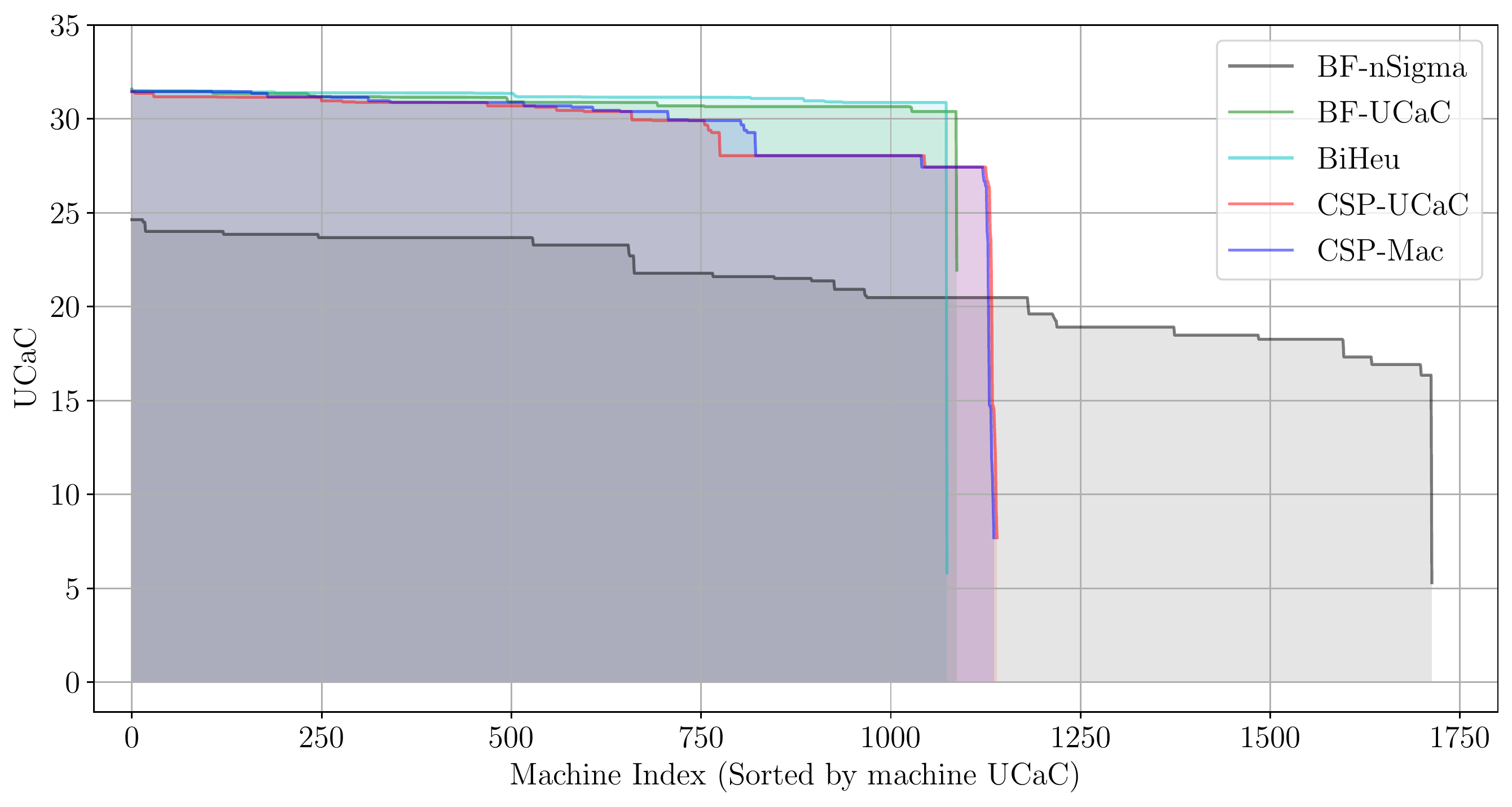}
      \caption{Empty cluster case}
      \label{subfig:empty}
    \end{subfigure}%
\caption{Comparison by visualizing UCaC on machines.}
\label{fig:bp-visual}
\end{figure}
\subsubsection{Scalability analysis.} 
In the real practice of SBPP, especially in the area of cloud service, one major concern is the solving time. Consider the algorithm complexity. Suppose $K$, $M$, $N$ are the number of services, the number of requested containers, and the number of machines in respective, and $|P|$ is the number of patterns used in CSP solvers. Best-Fit heuristics algorithms including BF-n$\sigma$ and BF-UCAC have complexity $O(MN)$, while BiHeu has $O(K\log K + N\log N + MN)$ where the first terms are for sorting services and machines respectively. For CSP methods, there are two computing phases -- columns generation and cutting stock problem optimization; among which the first phase depends on whether we have good predefined initial patterns as well as the generation algorithm which is hard to estimate, while the second phase's complexity is rough $O(a * |P|N)$ where $a$ depends on the used solver. 
In our experiments, with a platform of 8 cores of Intel Xeon E5-2690 and 128GiB memory, the run time is as follows. For the largest container requests $M=14,213$, BF-UCaC and BiHeu consume 14 and 0.7$s$ in respect, while BF-n$\sigma$ consumes 0.4$s$. With respect to CSP-UCaC, our most time-consuming experiments is for $K=20, \alpha=0.999$, where the pattern generation (column generation) with strict convergence consumes 14,697$s$, the cutting stock solving consumes 154$s$. However, for most cases, especially when $K \leq 15$, the pattern generation typically consumes tens of seconds, and the solving cutting stock consumes less than 1$s$. 

\subsection{Experiments on Real Traces}
\label{sub:experiments-real}
We evaluate the proposed methods on a 9-day regional trace data of a first-party application. The dataset contains 17 main services and over 10,000 containers at peak time in total; the details are given in Appendix \ref{app:dataset}. Since the focus of this paper is on the container allocation at usage peaks, we directly summarize the empirical peak resource usage distribution from data. Also, we exclude 2 weekend days in the 9-day period, as at weekends workloads are significantly low and no nontrivial container allocation was happened.

In all, we compare our proposed methods with two baselines (BF-n$\sigma$ and CSP-Mac) on container allocations of 7 days with only the first day initialized as an empty cluster. Three metrics are used for evaluation: UCaC at confidence $99.9\%$, the average \#machines, and total machine capacity violations in the 7-day period. We perform 5 tests with different random seeds and report the average values.

The results are shown in Table \ref{tab:experiments-real}.
For the BF-$n\sigma$ methods, the BF-$3.09\sigma$ uses more machines and has higher {\em UCaC} than all other methods. The performance of BF-$n\sigma$ can be improved by changing $n$ (or $\alpha$), and the best $n$ is 1.23 ($\alpha=0.89$) through our experiments, while the performance is still worse than {\em UCaC} based methods.

Compared to CSP-Mac, the best method that optimizing the number of used machines under chance constraints, our {\em UCaC}-based methods have significantly lower UCaC and even less number of used machines averaged on 7 days, with only a slight increase of cumulative violations in 7 days. For the better cumulative performance on the number of used machines, the primary reason is that UCaC-based methods explicitly consider the effect of non-empty machines while CSP-Mac can not do that. In summary, in practice of continuous container allocations our {\em UCaC}-based methods should be able offer better allocation strategies that can balance resource usage and violations than the baselines.

\begin{table}[]
\centering
\scalebox{0.75} 
{
    \begin{tabular}{l|c|c|c|}
    \toprule
        Algorithm & UCaC$_{99.9\%}$ & \#Machines & TotalViolations \\
        \midrule
        BF-$1.15\sigma$ & 41293.5 & 1305 & 4.4 \\
        BF-$1.23\sigma$ & 41558.8 & 1337 & 3.8 \\
        BF-$3.09\sigma$ & 46038.3 & 2144 & 0.0 \\
        CSP-Mac & 42582.0 & 1395 & 1.9 \\
        \midrule
        BF-UCaC & 41738.5 & 1350 & 2.6 \\
        BiHeu & 41450.4 & 1333 & 2.4 \\
        CSP-UCaC & 41402.3 & 1332 & 2.0 \\
    \bottomrule
    \end{tabular}
}
\caption{Performance of 7-day continuous experiments on a real trace, averaged on 5 simulations with different seeds.}
\label{tab:experiments-real}
\end{table}

\section{Related work}
\label{sec:related-work}

\textbf{Stochastic Bin Packing Problem.} Bin packing Problem (BPP) is a classic combinatorial optimization problem and has been thoroughly investigated through the past two decades; see \cite{delorme2016bin} and \cite{christensen2016multidimensional} for detailed review. Our work specifically focuses on the stochastic bin packing problem (SBPP), an extension of the classic BPP, where the item sizes are stochastic. The physical capacity constraint may be violated with a small probability, providing the chance of over-commitment. Note that the term SBPP was also used to denote other extensions to the BPP, such as the item profits when being taken \cite{perboli2012stochastic}, which is not our concerning scope. \citet{coffman1980stochastic} study a stochastic model in which an arbitrary distribution of item sizes is assumed. They obtained an asymptotic expected bin occupancy  for the next-fit algorithm. Different from this work which assumed the item size to be a known random variable, \citet{cohen2019} study the case that the item size follows unknown distribution and focus on computing a solution that is feasible with high probability before observing the actual sizes. They model it as chance constraints and propose a formulation that transforms each chance constraint into a sub-modular function. They developed a family of online algorithms and provided a constant factor guarantee from optimal. \citet{martinovic2021} derive several (improved) lower and upper bounds and study their worst-case performance metrics. They also describe various linearization techniques to overcome the drawback of nonlinearity in the original model proposed by \citet{cohen2019}. Besides,  \citet{kuccukyavuz2021chance} gives a broader survey on applications of chance-constrained optimization.

\textbf{Resource over-commitment and virtual machine consolidation in cloud.} Overcommitment [\citealp{jin2020improving,sun2018rose}] in which the sum of allocated resources to multiple containers is larger than physical capacity, is proposed to increase the resource utilization in recent cloud architecture like Borg \cite{borg}. With proper consolidation of multiple containers~\cite{borg-next}, statistical multiplexing to over-commit resources can avoid almost all resource violations. Such consolidation algorithms can also relate to research on virtual machine consolidation~\cite{roytman2013algorithm}, which searches the optimal combination of workloads or containers to avoid shared resource contention. Orthogonal to better scheduling, others focus on the prediction of container utilization. \citet{noman2021} propose peak oracle, which is the maximum idle capacity that a machine can advertise without violating any performance. Autopilot \cite{rzadca2020autopilot} uses a decision-focused prediction approach to do vertical scaling to automatically set the resource to match the container’s future peak. Besides, \citet{util-pred} is an earlier work that advises container scheduling based on prediction on container's future resource usage. 

\textbf{Adoption of Solvers in Resource allocation}
Recently, integrating solvers for mathematical programming into resource allocation has become popular in system communities \cite{narayanan2021solving, ras2021}.
\citet{narayanan2021solving} noticed that many allocation problems in the computer system are granular and proposed an efficient new method that splits the problem into
smaller problems that can be solved by exact solvers. RAS (\cite{ras2021}) is integrated with exact solvers to scale
resource allocation to problems with a larger size.

\section{Conclusions}
In this paper, we introduce and formulate the problem of stochastic bin packing on nonempty machines, which was not well realized in existing operational research. In particular, by introducing a new optimization metric {\em UCaC}, we propose a unified problem formulation for the SBPP with both empty and nonempty machines. Further, we designed solving algorithms of both heuristics and cutting-stock based exact algorithms. Extensive experiments on both synthetic and real cloud traces demonstrate that our {\em UCaC}-based optimization methodology performs better than existing approaches that optimize the number of used machines. Recently, integrating solvers for mathematical programming into scheduling has become popular in system communities \cite{narayanan2021solving, ras2021}. As part of this trend, this paper makes the very first step for the stochastic bin packing on nonempty machines, which is an important problem in cloud resource scheduling.

\begin{acks}
We thank Fangkai Yang, Jue Zhang, Tao Shen, Rahul Mohana Narayanamurthy, Terry Yang, Chetan Bansal, Senthil kumaran Baladhandayutham, Victor Rühle, Randy Lehner, Jim Kleewein, Silvia Yu and Andrew Zhou for the insights and discussions on this work and their support throughout the project.
\end{acks}

\bibliographystyle{ACM-Reference-Format}
\bibliography{references}

\clearpage
\appendix

\section{Proof of Theorem~\ref{th:bound}}
\label{app:proof-of-th5}

\begin{proof}
Let $N$ be the number of machines that the formulation with UCaC minimized purchases when serving all $J$ items.
For any machine $1 \leq i \leq N$, $S_i$ is defined to be the set of jobs assigned to machine $i$. 
For any pair of machines $i  \neq i_0 $, when the solution is optimal, the set $S_i \cup  S_{i_0}$ is infeasible for the modified capacity constraint, i.e., the items in machine  $S_i $ and $ S_{i_0}$  could not be assigned to one machine when the UCaC used is minimized. Otherwise, since $ \sqrt{\sum_{j \in S_i \cup {S_{i_0}}} b_j}< \sqrt{\sum_{j \in S_i} b_j}+\sqrt{\sum_{j \in S_{i_0}} b_j}$, we can place the items in $S_i $ and  $ S_{i_0}$ in machine $i$ (or $i_0$) and reduce the UCaC, which contradicts the assumption that the UCaC used is optimized. By taking $b_j{^*} = \frac{b_j D(\alpha)^2}{V^2 }, \mu_{j}{^*} = \frac{\mu_{j}}{V} $, the modified capacity constraint can be normalized as $\sum_{j=1}^{K}{x_{i, j}\mu_j^*} +  \sqrt{\sum_{j=1}^{K}{x_{i, j} b_j^*} } \leq 1 $.

\begin{lemma}
For any infeasible item set $T$, we have $\sum_{j\in{T}}(\mu_j^*+b_j^*)>\frac{3}{4}$
\end{lemma}

\begin{proof}{Proof}
For any infeasible set T, we have by definition $\sum_{j\in T}u_j^*+\sqrt{\sum_{j \in T}b_j^*}$ > 1. We denote $x=\sum_{j \in T} u_j^*$, and $y=\sqrt{\sum_{j\in T}b_j^*}$. Then  $y > 1-x$. If $x>1$, this lemma holds. Otherwise, we obtain :
\begin{equation}
    x+y^2>x+(1-x)^2=x^2-x+1=(x-\frac{1}{2})^2+\frac{3}{4}\geq \frac{3}{4}
\end{equation}
\end{proof}

Given $S_i \cup {S_{i_0}}$, we have $\sum_{j\in{S_i \cup {S_{i_0}}}}(\mu_j^*+b_j^*)>\frac{3}{4}$. By summing up this inequality for all pairs of machines i and $i_0$, we obtain:
\begin{equation}
    \sum_{1\leq i \leq N, 1\leq i_0 \leq N, i \neq i_0} \sum_{j \in S_i \cup S_{i_0}}(u_j^* +b_j^*) >\frac{3}{4}\tbinom{N}{2} \times 2
\end{equation}
 The left hand side of this inequality is equal to $2(N-1)\sum_{1\leq j \leq J }(u_j^* +b_j^*)$. This is because that for each machine $i$, it appears  $2(N-1)$ times, thus the items in each machine also appear $2(N-1)$ times. Therefore, 
\begin{equation}
    \frac{3}{8}N<\sum_{1 \leq j \leq J} (u_j^*+b_j^*)
\end{equation}

For the optimal assignment, we use $S_i^*$ to denote the set of containers in machine i. For containers in  $S_i^*$, we have $\sum_{j\in S_i^*}u_j^*+\sqrt{\sum_{j \in S_i^*}b_j^*}$ <1. Since $u_j^*$ >0 and $b_j^* >0$, we have that $\sum_{j \in S_i^*}b_j^*<1$. Therefore $\sum_{j\in S_i^*}u_j^*+\sum_{j \in S_i^*}b_j^* < \sum_{j\in S_i^*}u_j^*+\sqrt{\sum_{j \in S_i^*}b_j^*} <1$. Summing up for all of the machine $i$, we have:
\begin{equation}
  \sum_{1 \leq j \leq J} (u_j^*+b_j^*)=\sum_{1\leq i \leq N^*}(\sum_{j\in S_i^*}u_j^*b_j^*)< N^* 
\end{equation}
Therefore, we obtain that $ \frac{3}{8}N< N^*$, that is $N<\frac{8}{3}N^*$.
\end{proof}

\section{Details of Bi-Heuristics Algorithm}
Algorithm ~\ref{alg:solver-bilevel-heuristics} descibes the detailed procedures of our proposed online bi-heustics in Section~\label{sub:solve-offline-heuristic}.

\begin{algorithm}
\DontPrintSemicolon
\caption{Bi-level Heuristic for SBPP}
\label{alg:solver-bilevel-heuristics}

\SetAlgoVlined
\SetKwInOut{Input}{Input}
\SetKwInOut{Output}{Output}
\SetKw{KwBy}{by}

\Input{Request $m \in \I^K$}
\Input{Parameters of Equation \ref{eq:reformulation}: $\{ (\mu_j, b_j), z, V, \alpha\}$}
\Output{Mapping of containers to machines $x \in I^{N \times K}$}

\sf{
\Begin{
    \tcp{Presort}
    \nl Sort machines in non-increasing order of $B_i, ~ i \in [N]$; \\
    \nl Sort services in non-increasing order of $b_j / \mu_j, ~ j \in [K]$;\\
    
    \nl $r \gets m \in \I^K$; \\
    
    \tcp{Allocate machine for containers}
    \nl \ForEach{$i = 1 \to N$}{
        \nl \ForEach{$j = 1 \to K$}{
            \nl {$x_{i, j} \gets \max_{w \in I, 0 \leq w \leq r_j}{w}$, \\ 
                \quad \quad \quad s.t.~ $w \mu_j + D(\alpha) \sqrt{w b_j + B_j} \leq V - C_i$};\\
            \nl {$B_i \gets B_i + x_{i,j} b_j$};\\
            \nl {$C_i \gets C_i + x_{i,j} \mu_j$};\\
            \nl {$r_j \gets r_j - x_{i,j}$};\\
            
            \nl \If{$V-C_i \leq \sigma$}{
                \nl Break; \\
            }
        }
    }
}
}
\end{algorithm}

\section{Details of Section \ref{sub:solve-offline-cuttingstock}}
\label{app:solve-sbp-empty}

In this appendix, we show details of the cutting stock formulation and the column generation method to generate the pattern set.

\subsection{The cutting stock problem}
The cutting stock formulation (\ref{eq:cutting-stock}) can be used to solve the bin packing problem by finding a pattern set $P$ using the column generation technique \cite{delorme2016bin}. 
A pattern $p\in \R^K$ in  $P$ is a combination of items that satisfies the capacity constraint.
Let $P \in \R^{K \times L}$ be the pattern set containing $L$ patterns, $m\in \R^K$ be the requests of the $K$ services and $v_j$ be the number of times each pattern $p_j$ is used.

The cutting stock problem is defined as follows,
\begin{equation}
\label{eq:cutting-stock}
\begin{aligned}
\min_{v} & \sum_{j=1}^{L} v_j \\
\text{ s.t. } & \sum_{j=1}^{L} p_{kj} v_j \geq m_k, \quad k \in [K] \\
& v_j \in \I, \quad j \in [L], 
\end{aligned}
\end{equation}
where $P \in \R^{K \times L}$ is the pattern set, each column of $P$ is a single pattern and $v_j$ is the number of times pattern $p_j$ is used. 
For each bin, the items are allocated according to the pattern.
Thus, the number of patterns used ($\sum_{j=1}^{L} v_j$) is equivalent to the number of bins used.

The relaxation of the cutting stock problem is defined as
\begin{equation}
\label{eq:cutting-stock-relax}
\begin{aligned}
\min_{x} & \sum_{j=1}^{L} v_j \\
\text{ s.t. } & \sum_{j=1}^{L} p_{kj} v_j \geq m_k, \quad k \in [K] \\
& v_j \geq 0,  v_j \in \R, \quad j \in [L].
\end{aligned}
\end{equation}

\subsection{The column generation method}
The column generation method solves the relaxation of the cutting stock problem (\ref{eq:cutting-stock-relax}) and generates a pattern set $P$ iteratively.
The relaxation of the cutting stock problem (\ref{eq:cutting-stock-relax}) is called the restricted master problem (RMP) in the column generation method.

The dual of RMP is given by:
\begin{equation}
\label{eq:cutting-stock-relax-dual}
\begin{aligned}
\max_{\pi} ~& \pi ^\top m \\
\text{ s.t. } ~& \sum_{i=1}^{K}\pi_i p_{ij} \leq 1, \quad j \in [L] \\
& \pi_i \geq 0, \quad i \in [K].
\end{aligned}
\end{equation}
Let $\bar \pi$ be the optimal solution of the dual problem (\ref{eq:cutting-stock-relax-dual}).
A pattern $v$ is computed by solving the subproblem:
\begin{equation}
\label{eq:cutting-stock-sub}
\begin{aligned}
\min_{v} ~& c = 1 - \bar\pi^\top v \\
\text{ s.t. } ~& v \in \mathcal{F} \\
\end{aligned}
\end{equation}
where $\mathcal{F}$ is the set of feasible patterns.
\begin{remark}
In our setting, $\mathcal{F}$ is the set of patterns that satisfies the 
capacity constraint of CPU usage (\ref{eq:pattern1}).
The cloud provider can also add constraints of other resources (memory, etc.) into $\mathcal{F}$.
\end{remark}

The column generation method starts with a given pattern set $P^{(0)}$, which can be initialized as a diagonal matrix.
In our setting, $P^{(0)}_{ii}$ is set to the maximum number of containers of service $i$ that can fit into the machine.

At each iteration, the column generation method first solves the dual of RMP with the pattern set $P$.
Then, the subproblem (\ref{eq:cutting-stock-sub}) is solved using the solution from the dual.
If the reduced cost $c = 1 - \bar \pi^\top v^* < 0$, then the new pattern $v^*$ is added to the pattern set $P$.
The procedure is iterated until no new pattern with negative reduced cost can be found.

\section{Illustration of Cluster}
\label{app:cluster}

Fig.~\ref{fig:cluster} illustrates a simple cluster layout, where 6 containers of 3 services (say A, B, and C) are allocated to machine 1 and machine 2, while machine 3 is still empty (i.e., has no assigned containers).
\begin{figure} 
  \centering
  \includegraphics[width=.65\linewidth]{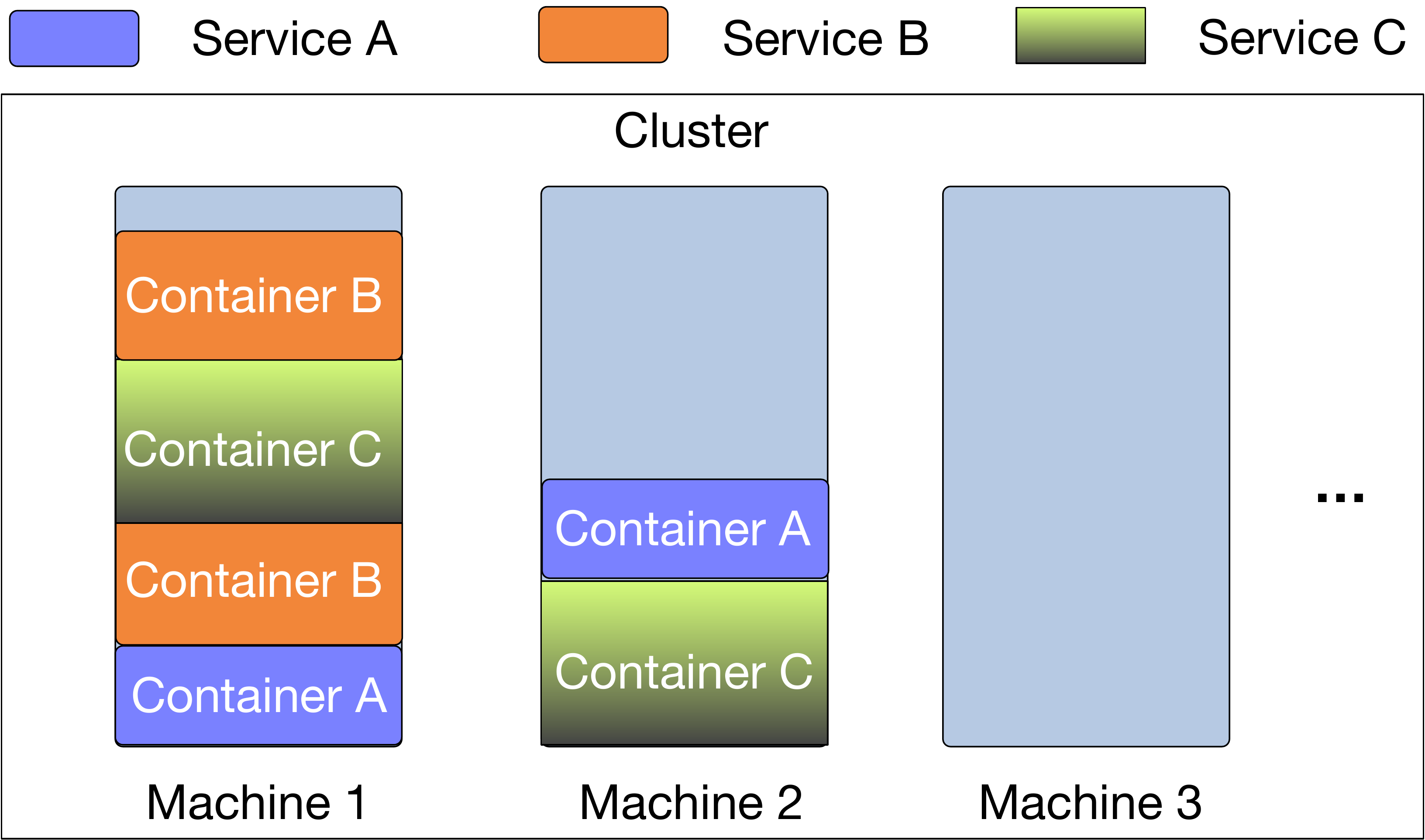}
  \caption{Abstract cluster, machines, containers and services.}
  \label{fig:cluster}
\end{figure}

\section{Dataset details}
\label{app:dataset}

\subsection{Real traces}
Our used dataset is sampled from a real workload which consists of 17 main services. For data privacy compliance, we can't release the full data. Instead, we give the summary of necessary statistics closely related to our experiments. Fig.~\ref{fig:peak_ts_stats} presents the mean cpu core usage of 17 services at {\em peak} time of each day, as well as the ratio of standard deviation against mean value of cpu core usage. As shown, for most services, the cpu usage on weekends is very low and the relative deviation is high, while in weekdays the mean cpu usage is high with a relatively reasonable deviation. Note that in our experiments of section \ref{sub:experiments-real} the container bin packing is only done on 7 weekdays.

\begin{figure}
\begin{subfigure}{.45\textwidth}
  \centering
  \includegraphics[width=.9\linewidth]{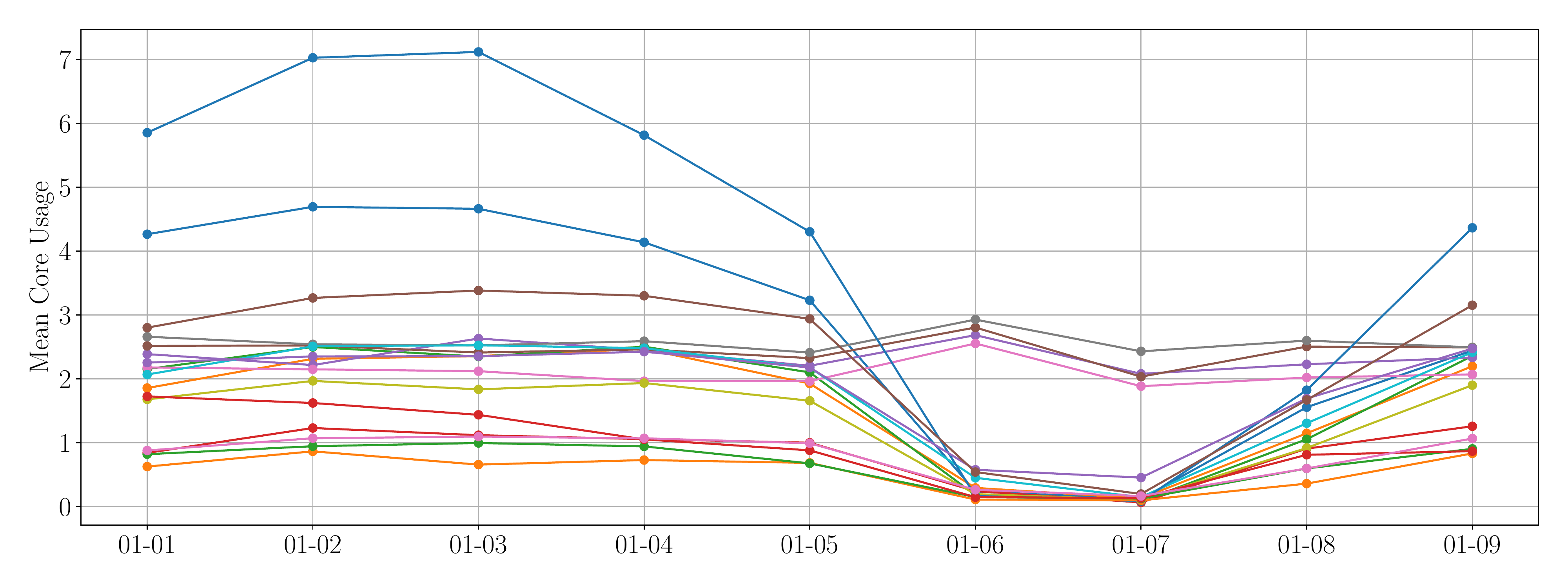}
  \caption{Mean CPU core usage at peaks}
  \label{fig:peak_ts_mean}
\end{subfigure}%
\\
\begin{subfigure}{.45\textwidth}
  \centering
  \includegraphics[width=.9\linewidth]{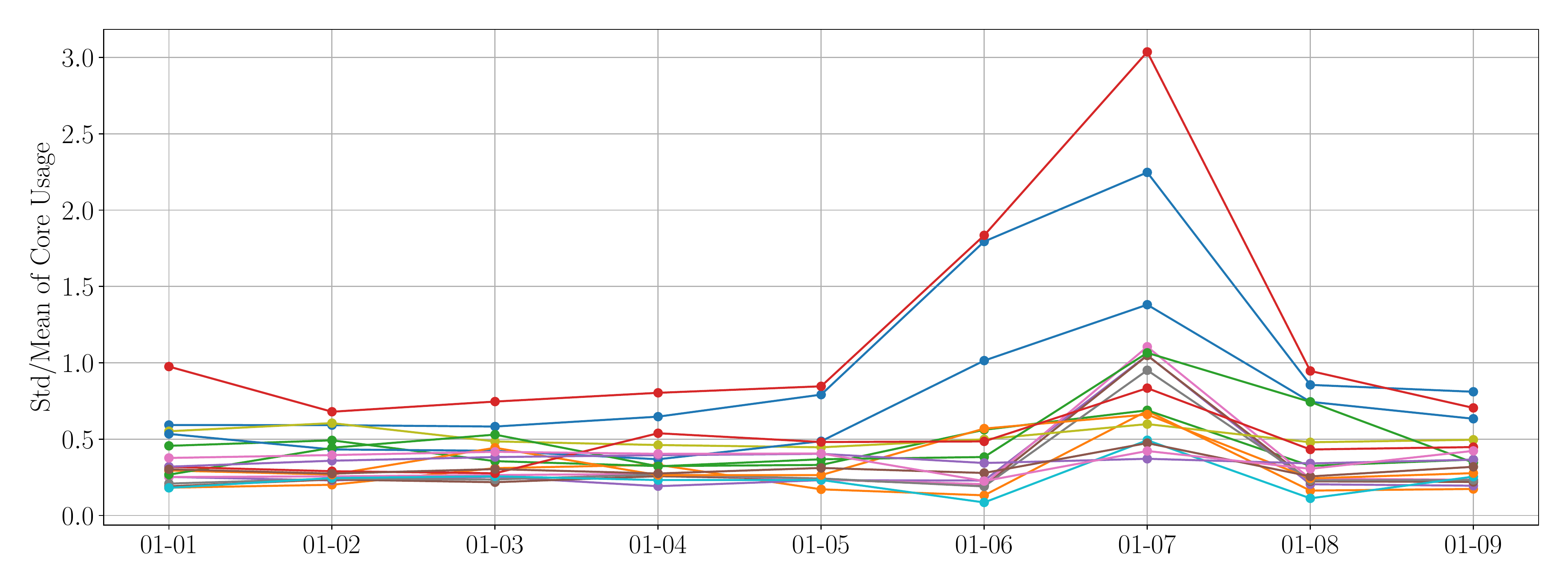}
  \caption{Ratio of Standard Deviation against Mean at peaks}
  \label{fig:peak_ts_ratio}
\end{subfigure}
\caption{Per-service statistics at daily Peaks in real traces}
\label{fig:peak_ts_stats}
\end{figure}

\begin{table}
    \centering
    \scalebox{0.75} 
    {
    \begin{tabular}{c|ccccccccc}
        \toprule
        Mean &  6.18 & 2.47 & 1.07 & 4.12 & 1.06 & 0.73 & 1.94 & 2.48 & 2.42\\
        Std & 1.73 & 0.47 & 0.43 & 2.69 & 0.85 & 0.19 & 0.9 & 0.82 & 0.97\\
        \#Containers & 270 & 55 & 1618 & 904 & 576 & 1085 & 1035 & 118 & 1450\\
        \midrule
        Remove Rate & 0.5 & 0.3 & 0.8 & 0.5 & 0.8 & 0.8 & 0.5 & 0.5 & 0.5\\
        \midrule
        \midrule
        Mean &  2.49 & 0.97 & 2.46 & 2.52 & 1.06 & 2.59 & 1.96 & 3.33 & \\
        Std & 0.62 & 0.31 & 0.62 & 0.84 & 0.57 & 0.7 & 0.55 & 0.9 & \\
        \#Containers & 313 & 44 & 544 & 697 & 427 & 363 & 360 & 701 & \\
        \midrule
        Remove Rate & 0.5 & 0.8 & 0.3 & 0.5 & 0.8 & 0.3 & 0.3 & 0.5 & \\
        \bottomrule
    \end{tabular}
    }
    \caption{Per-service statistics in the synthetic dataset.}
    \label{tab:synthetic_stats}
\end{table}

\subsection{Synthetic datasets}
\label{sub:synthetic_datasets}
The base per-service container distributions and numbers are summarized from a weekday of the real trace, {\em with some re-scaling for data privacy requirements}. Their statistics are given in Table~\ref{tab:synthetic_stats} where each column is for a service. Besides, the table also provides other parameters for simulation, including the number of containers per service in the base layout, and the rates that used to remove containers per service to construct the initial nonempty layout.

\end{document}